\documentclass[reqno]{amsart}%
\usepackage{amsfonts}
\usepackage{amsmath}
\usepackage{amssymb}
\usepackage{graphicx}%
\providecommand{\U}[1]{\protect\rule{.1in}{.1in}}
\newtheorem{theorem}{Theorem}
\theoremstyle{plain}

\newtheorem{corollary}{Corollary}

\newtheorem{definition}{Definition}
\newtheorem{example}{Example}

\newtheorem{remark}{Remark}

\numberwithin{equation}{section}
\begin{document}
\title[Sherman's inequality and its converse for strongly convex functions]{Sherman's inequality and its converse for strongly convex functions with
applications to generalized $f$-divergences}
\author{Slavica Iveli\'{c} Bradanovi\'{c}$^{1}$}
\address{1 -Faculty of Civil Engineering, Architecture And Geodesy, University of
Split, Matice Hrvatske 15, 21000 Split, Croatia}
\date{}
\maketitle

\begin{abstract}
Considering the weighted concept of majorization, Sherman obtained
generalization of majorization inequality for convex functions known as
Sherman's inequality. We extend Sherman's result to the class of $n$-strongly
convex functions using extended idea of convexity to the class of strongly
convex functions. We also obtaine upper bound for Sherman's inequality, so
called the converse Sherman inequality, and as easy consequences we get
Jensen's as well as majorization inequality and their conversions for strongly
convex functions. Obtained results are stronger versions for analogous results
for convex functions. As applications, we introduced a generalized concept of
$f$-divergence and derived some reverse relations for such concept.

\end{abstract}

\section{Introduction}

A function $f:[\alpha,\beta]\rightarrow\mathbb{R}$ is called strongly convex
with modulus $c>0$ if
\begin{equation}
f(\lambda x+(1-\lambda)y)\leq\lambda f(x)+(1-\lambda)f(y)-c\lambda
(1-\lambda)(x-y)^{2} \label{Strong_c}%
\end{equation}
for all $x,y\in\lbrack\alpha,\beta]$ and $\lambda\in\lbrack0,1].$

The concept of strongly convexity has been introduced by Polyak \cite{P}. It
has a large number of appearance in many different fields of applications,
particular in many branches of mathematics as well as optimization theory,
mathematical economics and approximation theory. Strongly convex functions
have many nice properties (see \cite{N}).

A function $f$ that satisfies (\ref{Strong_c}) with $c=0,$ i.e.
\begin{equation}
f(\lambda x+(1-\lambda)y)\leq\lambda f(x)+(1-\lambda)f(y) \label{Conv}%
\end{equation}
is convex in usual sense. Specially, if the inequality in (\ref{Conv}) is
strict, then $f$ is called strictly convex.

It is well known that following implications hold:
\[
\text{strongly convex \ \ }\Rightarrow\text{\ \ \ strictly convex
\ \ }\Rightarrow\text{\ \ \ convex.}%
\]
But the reverse implications are not true, in general.

\begin{example}
The function $f(x)=x^{2}$ is strongly convex and also strictly convex and
convex. The gunction $g(x)=e^{x}$ is strictly convex and convex but not
strongly convex. The function $h(x)=x$ is convex but not strictly neither
strongly convex.
\end{example}

In the theory of convex functions, natural generalization are convex functions
of higher order, i.e. $n$-convex functions. The notion of $n$-convexity was
defined in terms of divided differences by T. Popoviciu \cite{POP}\ which we
introduce in the sequel.

A function $f:[\alpha,\beta]\rightarrow\mathbb{R}$ is said to be $n$-convex if
for every choice of $n+1$ distinct points $z_{0},...,z_{n}\in\lbrack
\alpha,\beta],$ the $n$th order divided difference is nonnegative, i.e.%
\begin{equation}
\lbrack z_{0},z_{1},...,z_{n};f]\geq0, \label{n_con}%
\end{equation}
where divided difference may be formally defined by%
\begin{align*}
\left[  z_{i};f\right]   &  =f(z_{i}),\text{ \ \ }i=0,...,n\\
\lbrack z_{0},...,z_{n};f]  &  =\frac{[z_{1},...,z_{n};f]-[z_{0}%
,...,z_{n-1};f]}{z_{n}-z_{0}}\text{ .}%
\end{align*}
The value $[z_{0},...,z_{n};f]$ is independent of the order of the points
$z_{0},...,z_{n}.$ This definition may be extended to include the case in
which some or all the points coincide. Assuming that $f^{(j-1)}(z)$ exists, we
define%
\begin{equation}
\lbrack\underset{j\text{-times}}{\underbrace{z,...,z}};f]=\frac{f^{(j-1)}%
(z)}{(j-1)!}. \label{dd4}%
\end{equation}

\begin{remark}
It is known that $1$-convex function is increasing function and $2$-convex
function is just ordinary convex function, i.e. convex in usual sense.
\newline If $f^{(n)}$ exists, then $f$ is $n$-convex iff $f^{(n)}\geq0.$
\newline Also, if $f$ is $n$-convex for $n\geq2$, then $f^{(k)}$ exists and
$f$ is $(n-k)$-convex for $1\leq k\leq n-2.$ For more information see
\cite{PPT}.
\end{remark}

Following R. Gera and K. Nikodem \cite{GN}, we say that a function
$f:[\alpha,\beta]\rightarrow\mathbb{R}$ is strongly convex of order $n$ with
modulus $c>0$ (or $n$-strongly convex with modulus $c>0$)\ if
\begin{equation}
\lbrack z_{0},...,z_{n};f]\geq c \label{n_str_con}%
\end{equation}
for all $z_{0},...,z_{n}\in\lbrack\alpha,\beta].$

\begin{remark}
Note that $2$-strongly convex function with modulus $c$ is just strongly
convex function with modulus $c$ as given by (\ref{Strong_c}). \newline For
$n=2,$ the condition (\ref{n_str_con}) is equivalent to%
\[
\frac{f(z_{0})}{(z_{0}-z_{1})(z_{0}-z_{2})}+\frac{f(z_{1})}{(z_{1}%
-z_{2})(z_{1}-z_{0})}+\frac{f(z_{2})}{(z_{2}-z_{1})(z_{2}-z_{0})}\geq c
\]
or%
\[
f(z_{1})\leq\frac{z_{2}-z_{1}}{z_{2}-z_{0}}f(z_{0})+\frac{z_{1}-z_{0}}%
{z_{2}-z_{0}}f(z_{2})-c(z_{2}-z_{1})(z_{1}-z_{0}).
\]
A function $f:[\alpha,\beta]\rightarrow\mathbb{R}$ is a strongly $n$-convex
with modulus $c$ iff the function $g(x)=f(x)-cx^{n}$ is $n$-convex.\newline A
function $f:[\alpha,\beta]\rightarrow\mathbb{R}$ is a strongly $n$-convex with
modulus $c$ iff $f^{(n)}\geq cn!.$\newline For more information see \cite{GN},
\cite{N}, \cite{NP}.
\end{remark}

The concept of strongly convexity is a strengthening of the concept of
convexity and some properties of strongly convex functions are just stronger
versions of analogous properties of convex functions.

For $f:[\alpha,\beta]\rightarrow\mathbb{R}$ strongly convex function with
modulus $c>0,$ Jensen's inequality%
\begin{equation}
f\left(
%TCIMACRO{\dsum \limits_{i=1}^{m}}%
%BeginExpansion
{\displaystyle\sum\limits_{i=1}^{m}}
%EndExpansion
a_{i}x_{i}\right)  \leq%
%TCIMACRO{\dsum \limits_{i=1}^{m}}%
%BeginExpansion
{\displaystyle\sum\limits_{i=1}^{m}}
%EndExpansion
a_{i}f(x_{i})-c%
%TCIMACRO{\dsum \limits_{i=1}^{m}}%
%BeginExpansion
{\displaystyle\sum\limits_{i=1}^{m}}
%EndExpansion
a_{i}(x_{i}-\bar{x})^{2} \label{Jen_strong}%
\end{equation}
holds, where $\mathbf{x}=(x_{1},...,x_{m})\in\lbrack\alpha,\beta]^{m},$
$\mathbf{a}=(a_{1},...,a_{m})\in\lbrack0,\infty)^{m}$ with $\sum
\nolimits_{i=1}^{m}a_{i}=1$ and $\bar{x}=\sum\nolimits_{i=1}^{m}a_{i}x_{i}$
(see \cite{MN})$.$ On the other side, Jensen's inequality for a classical
convex function $f$ has the form%
\begin{equation}
f\left(
%TCIMACRO{\dsum \limits_{i=1}^{m}}%
%BeginExpansion
{\displaystyle\sum\limits_{i=1}^{m}}
%EndExpansion
a_{i}x_{i}\right)  \leq%
%TCIMACRO{\dsum \limits_{i=1}^{m}}%
%BeginExpansion
{\displaystyle\sum\limits_{i=1}^{m}}
%EndExpansion
a_{i}f(x_{i}). \label{Jen_conv}%
\end{equation}
If we compare (\ref{Jen_strong}) with (\ref{Jen_conv}), note that the
inequality (\ref{Jen_strong}) includes a better upper bound for $f\left(
%TCIMACRO{\tsum \nolimits_{i=1}^{m}}%
%BeginExpansion
{\textstyle\sum\nolimits_{i=1}^{m}}
%EndExpansion
a_{i}x_{i}\right)  $ since $c%
%TCIMACRO{\tsum \nolimits_{i=1}^{m}}%
%BeginExpansion
{\textstyle\sum\nolimits_{i=1}^{m}}
%EndExpansion
a_{i}(x_{i}-\bar{x})\geq0.$ Since specially for $c=0$ the strogly convexity
reduces to the ordinary convexity, then (\ref{Jen_strong}) becomes
(\ref{Jen_conv}).

Closely connected to Jensen's inequality (\ref{Jen_conv}) is the
Lah-Ribari\v{c} inequality%
\begin{equation}
\sum_{i=1}^{m}a_{i}f\left(  x_{i}\right)  \leq\frac{\beta-\bar{x}}%
{\beta-\alpha}f(\alpha)+\frac{\bar{x}-\alpha}{\beta-\alpha}f(\beta)
\label{ineq_LahRib}%
\end{equation}
which holds for every convex function $f:[\alpha,\beta]\rightarrow\mathbb{R}$
and $\mathbf{x}=(x_{1},...,x_{m})\in\lbrack\alpha,\beta]^{m},$ $\mathbf{a}%
=(a_{1},...,a_{n})\in\left[  0,\infty\right)  ^{m}$ with $\sum_{i=1}^{m}%
a_{i}=1$ and $\bar{x}=%
%TCIMACRO{\tsum \nolimits_{i=1}^{m}}%
%BeginExpansion
{\textstyle\sum\nolimits_{i=1}^{m}}
%EndExpansion
a_{i}x_{i}$ (see \cite{LR}). The Lah-Ribari\v{c} inequality gives the upper
bound for the term $%
%TCIMACRO{\tsum \nolimits_{i=1}^{m}}%
%BeginExpansion
{\textstyle\sum\nolimits_{i=1}^{m}}
%EndExpansion
a_{i}f(x_{i})$ and often called the converse Jensen inequality.

\section{Preliminaries}

For two vectors $\mathbf{x}=(x_{1},...,x_{m}),\mathbf{y}=(y_{1},...,y_{m}%
)\in\lbrack\alpha,\beta]^{m},$ let $x_{[i]},y_{[i]}$ denote their increasing
order. We say that $\mathbf{x}$ majorizes $\mathbf{y}$ or $\mathbf{y}$ is
majorized by $\mathbf{x}$ and write%
\[
\mathbf{y}\prec\mathbf{x}%
\]
if%
\begin{equation}%
%TCIMACRO{\dsum \limits_{i=1}^{k}}%
%BeginExpansion
{\displaystyle\sum\limits_{i=1}^{k}}
%EndExpansion
y_{[i]}\leq%
%TCIMACRO{\dsum \limits_{i=1}^{k}}%
%BeginExpansion
{\displaystyle\sum\limits_{i=1}^{k}}
%EndExpansion
x_{[i]},\text{ \ \ }k=1,....,m, \label{maj}%
\end{equation}
with equality in (\ref{maj})\ for $k=m.$

The term majorization is introduced in the space $\mathbb{R}^{m}$, in which
the order is not defined, to compare and detect potential links between
vectors. The majorization relation is reflexive and transitive but it is not
antisymmetric (see \cite[p.~79]{MO}) and hence is a preordering not a partial
ordering. The majorization preorder on vectors is known as vector majorization
or classical majorization. This classical concept was initially studied by
Hardy et al. \cite{HLP}. A superb reference on the subject is \cite{MO}.

It is well known that
\[
\mathbf{y\prec x}\text{ \ \ iff \ \ }\mathbf{y}=\mathbf{xA}%
\]
for some doubly stochastic matrix $\mathbf{A}=(a_{ij})\in\mathcal{M}%
_{mm}(\mathbb{R})$, i.e. a matrix with nonnegative entries and rows and
columns sums equal to $1$.

Moreover, $\mathbf{y\prec x}$ implies
\[%
%TCIMACRO{\dsum \limits_{i=1}^{m}}%
%BeginExpansion
{\displaystyle\sum\limits_{i=1}^{m}}
%EndExpansion
f(y_{i})\leq%
%TCIMACRO{\dsum \limits_{i=1}^{m}}%
%BeginExpansion
{\displaystyle\sum\limits_{i=1}^{m}}
%EndExpansion
f(x_{i})
\]
for every continuous convex function $f:[\alpha,\beta]\rightarrow\mathbb{R}.$
This result, obtained by Hardy et al. \cite{HLP}, is well known as
majorization inequality and plays an important tool in the study of
majorization theory.

S. Sherman \cite{S} considered the weighted concept of majorization between
two vectors $\mathbf{x}=(x_{1},...,x_{m})\in\lbrack\alpha,\beta]^{m}$ and
$\mathbf{y}=(y_{1},...,y_{l})\in\lbrack\alpha,\beta]^{l}$ with nonnegative
weights $\mathbf{a}=(a_{1},...,a_{m})$ and $\mathbf{b}=(b_{1},...,b_{l})$. The
concept of weighted majorization is defined by assumption of existence of row
stochastic matrix $\mathbf{A}=(a_{ij})\in\mathcal{M}_{lm}(\mathbb{R}),$ i.e.
matrix with nonnegative entries and rows sums equal to $1$, such that
\begin{align}
a_{j}  &  =%
%TCIMACRO{\dsum \limits_{i=1}^{m}}%
%BeginExpansion
{\displaystyle\sum\limits_{i=1}^{m}}
%EndExpansion
b_{j}a_{ij},\text{ \ \ }j=1,...,l,\label{sh_assump_m}\\
y_{i}  &  =%
%TCIMACRO{\dsum \limits_{j=1}^{l}}%
%BeginExpansion
{\displaystyle\sum\limits_{j=1}^{l}}
%EndExpansion
x_{j}a_{ij},\text{ \ \ }i=1,...,m.\text{ }\nonumber
\end{align}
Sherman proved that under conditions (\ref{sh_assump_m}), the inequality
\begin{equation}%
%TCIMACRO{\dsum \limits_{i=1}^{m}}%
%BeginExpansion
{\displaystyle\sum\limits_{i=1}^{m}}
%EndExpansion
b_{i}f(y_{i})\leq%
%TCIMACRO{\dsum \limits_{j=1}^{l}}%
%BeginExpansion
{\displaystyle\sum\limits_{j=1}^{l}}
%EndExpansion
a_{j}f(x_{j}) \label{Sh_ineq}%
\end{equation}
holds for every convex function $f:[\alpha,\beta]\rightarrow\mathbb{R}$.

We can write the conditions (\ref{sh_assump_m}) in the matrix form%
\begin{equation}
\mathbf{a}=\mathbf{bA}\text{ \ \ and \ \ }\mathbf{y}=\mathbf{xA}^{T}\mathbf{,}
\label{sh_assump}%
\end{equation}
where $\mathbf{A}^{T}$ denotes transpose matrix.

In the sequel, we write%
\[
(\mathbf{y,b})\prec(\mathbf{x,a})
\]
and say that a pair $(\mathbf{y,b})$ is weighted majorized by $(\mathbf{x,a})$
if vectors $\mathbf{x,y}$ and corresponding weights $\mathbf{a,b}$ are such
that satisfy conditions (\ref{sh_assump_m}) for some row stochastic matrix
$\mathbf{A}.$

Sherman's generalization contains Jensen's as well as Majorization's
inequality as special cases as we pointed in the next remark.

\begin{remark}
a) For $m=1$ and $\mathbf{b}=[1],$ Sherman's inequality (\ref{Sh_ineq})
reduces to Jensen's inequality (\pageref{Jen_conv}). \newline b) For $m=l$ and
$\mathbf{b}=\mathbf{e}=(1,...,1),$ because $\mathbf{y}\prec\mathbf{x}$ gives
$\mathbf{y}=\mathbf{xA}^{T}$ with some doubly stochastic matrix $\mathbf{A}$
and $\mathbf{a}=\mathbf{bA=e}$, from Sherman's inequality (\ref{Sh_ineq}) we
get majorization inequality%
\begin{equation}%
%TCIMACRO{\dsum \limits_{i=1}^{m}}%
%BeginExpansion
{\displaystyle\sum\limits_{i=1}^{m}}
%EndExpansion
f(y_{i})\leqslant%
%TCIMACRO{\dsum \limits_{i=1}^{m}}%
%BeginExpansion
{\displaystyle\sum\limits_{i=1}^{m}}
%EndExpansion
f(x_{i}). \label{Maj1}%
\end{equation}
c)\ When $m=l,$ and all weights $b_{i}$ and $a_{j}$ are equal, the condition
$\mathbf{a}=\mathbf{bA}$ assures the stochastically on columns, so in that
case we deal with doubly stochastic matrices. Moreover, Sherman's inequality
(\ref{Sh_ineq}) reduces to
\begin{equation}%
%TCIMACRO{\dsum \limits_{i=1}^{m}}%
%BeginExpansion
{\displaystyle\sum\limits_{i=1}^{m}}
%EndExpansion
a_{i}f(y_{i})\leqslant%
%TCIMACRO{\dsum \limits_{i=1}^{m}}%
%BeginExpansion
{\displaystyle\sum\limits_{i=1}^{m}}
%EndExpansion
a_{i}f(x_{i}), \label{Maj2}%
\end{equation}
known as Fuchs' inequality (see \cite{F}).
\end{remark}

In recent times, Sherman's result has attracted the interest of several
mathematicians (see \cite{AKIBP1}-\cite{Bu}, \cite{IBNP}-\cite{IBP},
\cite{NI}-\cite{MN3}).

This paper is organized as follows. In Section 3 we obtain the Lah-Ribarich
inequality for strongly convex functions. We deal with Sherman's inequality
and its converse for strongly convex function. As easy consequences, we get
Jensen's and majorization inequalities and their conversions for strongly
convex functions. In Section 4, we obtain some inequalities for generalized
concept of $f$-divergence. In the last section, we extend Sherman's result to
the class of strongly convex functions of higher order.

\section{Sherman's type inequalities and conversions}

We start with the Lah-Ribarich inequality for strongly convex functions.

\begin{theorem}
Let $\mathbf{x}=(x_{1},...,x_{l})\in\lbrack\alpha,\beta]^{l}$ and
$\mathbf{a}=(a_{1},...,a_{l})\in\lbrack0,\infty)^{l}$ with $\sum
\nolimits_{j=1}^{l}a_{j}=1$ and $\bar{x}=\sum\nolimits_{j=1}^{l}a_{j}x_{j}.$
If $f:[\alpha,\beta]\rightarrow\mathbb{R}$ is strongly convex with modulus
$c>0,$ then%
\begin{equation}
\sum_{j=1}^{l}a_{j}f(x_{j})\leq\frac{\beta-\bar{x}}{\beta-\alpha}%
f(\alpha)+\frac{\bar{x}-\alpha}{\beta-\alpha}f(\beta)-c%
%TCIMACRO{\dsum \limits_{j=1}^{l}}%
%BeginExpansion
{\displaystyle\sum\limits_{j=1}^{l}}
%EndExpansion
a_{j}(\beta-x_{j})(x_{j}-\alpha). \label{LR_str}%
\end{equation}

\begin{proof}
Since for strongly convex function we have
\[
f(z_{1})\leq\frac{z_{2}-z_{1}}{z_{2}-z_{0}}f(z_{0})+\frac{z_{1}-z_{0}}%
{z_{2}-z_{0}}f(z_{2})-c(z_{2}-z_{1})(z_{1}-z_{0}),
\]
by substituting $z_{1}=x_{j},$ $z_{2}=\beta$ and $z_{1}=\alpha,$ we get%
\[
f(x_{j})\leq\frac{\beta-x_{j}}{\beta-\alpha}f(\alpha)+\frac{x_{j}-\alpha
}{\beta-\alpha}f(\beta)-c(\beta-x_{j})(x_{j}-\alpha).
\]
Now, multiplying with $a_{j}$ and summing over $j$ we have%
\[
\sum_{j=1}^{l}a_{j}f(x_{j})\leq\frac{\beta-%
%TCIMACRO{\dsum \limits_{j=1}^{l}}%
%BeginExpansion
{\displaystyle\sum\limits_{j=1}^{l}}
%EndExpansion
a_{j}x_{j}}{\beta-\alpha}f(\alpha)+\frac{%
%TCIMACRO{\dsum \limits_{j=1}^{l}}%
%BeginExpansion
{\displaystyle\sum\limits_{j=1}^{l}}
%EndExpansion
a_{j}x_{j}-\alpha}{\beta-\alpha}f(\beta)-c%
%TCIMACRO{\dsum \limits_{j=1}^{l}}%
%BeginExpansion
{\displaystyle\sum\limits_{j=1}^{l}}
%EndExpansion
a_{j}(\beta-x_{j})(x_{j}-\alpha)
\]
what we need to prove.
\end{proof}
\end{theorem}

Now we give Sherman's inequality for strongly convex functions.

\begin{theorem}
\label{Th_Sherman_strong}Let $\mathbf{x}=(x_{1},...,x_{l})\in\lbrack
\alpha,\beta]^{l},$ $\mathbf{y}=(y_{1},...,y_{m})\in\lbrack\alpha,\beta]^{m},$
$\mathbf{a}=(a_{1},...,a_{l})\in\lbrack0,\infty)^{l}$ and $\mathbf{b}%
=(b_{1},...,b_{m})\in\lbrack0,\infty)^{m}$ be such that $(\mathbf{y,b}%
)\prec(\mathbf{x,a}).$ Then for every $f:[\alpha,\beta]\rightarrow\mathbb{R}$
strongly convex with modulus $c>0,$ we have%
\begin{equation}%
%TCIMACRO{\dsum \limits_{i=1}^{m}}%
%BeginExpansion
{\displaystyle\sum\limits_{i=1}^{m}}
%EndExpansion
b_{i}f(y_{i})\leq%
%TCIMACRO{\dsum \limits_{j=1}^{l}}%
%BeginExpansion
{\displaystyle\sum\limits_{j=1}^{l}}
%EndExpansion
a_{j}f(x_{j})-c\left(
%TCIMACRO{\dsum \limits_{j=1}^{l}}%
%BeginExpansion
{\displaystyle\sum\limits_{j=1}^{l}}
%EndExpansion
a_{j}x_{j}^{2}-%
%TCIMACRO{\dsum \limits_{i=1}^{m}}%
%BeginExpansion
{\displaystyle\sum\limits_{i=1}^{m}}
%EndExpansion
b_{i}y_{i}^{2}\right)  . \label{Sh_str}%
\end{equation}

\begin{proof}
Using (\ref{sh_assump_m}) and applying (\ref{Jen_strong}), we have
\begin{align}%
%TCIMACRO{\dsum \limits_{i=1}^{m}}%
%BeginExpansion
{\displaystyle\sum\limits_{i=1}^{m}}
%EndExpansion
b_{i}f(y_{i})  &  =%
%TCIMACRO{\dsum \limits_{i=1}^{m}}%
%BeginExpansion
{\displaystyle\sum\limits_{i=1}^{m}}
%EndExpansion
b_{i}f\left(
%TCIMACRO{\dsum \limits_{j=1}^{l}}%
%BeginExpansion
{\displaystyle\sum\limits_{j=1}^{l}}
%EndExpansion
x_{j}a_{ij}\right) \label{Sh_1}\\
&  \leq%
%TCIMACRO{\dsum \limits_{i=1}^{m}}%
%BeginExpansion
{\displaystyle\sum\limits_{i=1}^{m}}
%EndExpansion
b_{i}\left(
%TCIMACRO{\dsum \limits_{j=1}^{l}}%
%BeginExpansion
{\displaystyle\sum\limits_{j=1}^{l}}
%EndExpansion
a_{ij}f(x_{j})-c%
%TCIMACRO{\dsum \limits_{j=1}^{l}}%
%BeginExpansion
{\displaystyle\sum\limits_{j=1}^{l}}
%EndExpansion
a_{ij}(x_{j}-y_{i})^{2}\right) \nonumber\\
&  =%
%TCIMACRO{\dsum \limits_{j=1}^{l}}%
%BeginExpansion
{\displaystyle\sum\limits_{j=1}^{l}}
%EndExpansion
a_{j}f(x_{j})-c%
%TCIMACRO{\dsum \limits_{i=1}^{m}}%
%BeginExpansion
{\displaystyle\sum\limits_{i=1}^{m}}
%EndExpansion
b_{i}%
%TCIMACRO{\dsum \limits_{j=1}^{l}}%
%BeginExpansion
{\displaystyle\sum\limits_{j=1}^{l}}
%EndExpansion
a_{ij}(x_{j}-y_{i})^{2}.\nonumber
\end{align}
By an easy calculation, we get
\begin{align}
&
%TCIMACRO{\dsum \limits_{j=1}^{l}}%
%BeginExpansion
{\displaystyle\sum\limits_{j=1}^{l}}
%EndExpansion
a_{j}f(x_{j})-c%
%TCIMACRO{\dsum \limits_{i=1}^{m}}%
%BeginExpansion
{\displaystyle\sum\limits_{i=1}^{m}}
%EndExpansion
b_{i}%
%TCIMACRO{\dsum \limits_{j=1}^{l}}%
%BeginExpansion
{\displaystyle\sum\limits_{j=1}^{l}}
%EndExpansion
a_{ij}(x_{j}-y_{i})^{2}\label{Sh_2}\\
&  =%
%TCIMACRO{\dsum \limits_{j=1}^{l}}%
%BeginExpansion
{\displaystyle\sum\limits_{j=1}^{l}}
%EndExpansion
a_{j}f(x_{j})-c%
%TCIMACRO{\dsum \limits_{i=1}^{m}}%
%BeginExpansion
{\displaystyle\sum\limits_{i=1}^{m}}
%EndExpansion
b_{i}%
%TCIMACRO{\dsum \limits_{j=1}^{l}}%
%BeginExpansion
{\displaystyle\sum\limits_{j=1}^{l}}
%EndExpansion
a_{ij}(x_{j}^{2}-2x_{j}y_{i}+y_{i}^{2})\nonumber\\
&  =%
%TCIMACRO{\dsum \limits_{j=1}^{l}}%
%BeginExpansion
{\displaystyle\sum\limits_{j=1}^{l}}
%EndExpansion
a_{j}f(x_{j})-c\left(
%TCIMACRO{\dsum \limits_{j=1}^{l}}%
%BeginExpansion
{\displaystyle\sum\limits_{j=1}^{l}}
%EndExpansion
a_{j}x_{j}^{2}-%
%TCIMACRO{\dsum \limits_{i=1}^{m}}%
%BeginExpansion
{\displaystyle\sum\limits_{i=1}^{m}}
%EndExpansion
b_{i}y_{i}^{2}\right)  .\nonumber
\end{align}
Now, combining (\ref{Sh_1}) and (\ref{Sh_2}), we get (\ref{Sh_str}).
\end{proof}
\end{theorem}

\begin{remark}
If we compare (\ref{Sh_str}) with (\ref{Sh_ineq}), note that the inequality
(\ref{Sh_str}) includes a better upper bound for $%
%TCIMACRO{\tsum \nolimits_{i=1}^{m}}%
%BeginExpansion
{\textstyle\sum\nolimits_{i=1}^{m}}
%EndExpansion
b_{i}f(y_{i})$ since $c\left(
%TCIMACRO{\tsum \nolimits_{j=1}^{l}}%
%BeginExpansion
{\textstyle\sum\nolimits_{j=1}^{l}}
%EndExpansion
a_{j}x_{j}^{2}-%
%TCIMACRO{\tsum \nolimits_{i=1}^{m}}%
%BeginExpansion
{\textstyle\sum\nolimits_{i=1}^{m}}
%EndExpansion
b_{i}y_{i}^{2}\right)  \geq0$ because $t\mapsto t^{2}$ is convex function and
then by Sherman's inequality we have $%
%TCIMACRO{\tsum \nolimits_{j=1}^{l}}%
%BeginExpansion
{\textstyle\sum\nolimits_{j=1}^{l}}
%EndExpansion
a_{j}x_{j}^{2}-%
%TCIMACRO{\tsum \nolimits_{i=1}^{m}}%
%BeginExpansion
{\textstyle\sum\nolimits_{i=1}^{m}}
%EndExpansion
b_{i}y_{i}^{2}\geq0.$ Moreover, we get the double inequality
\begin{align}%
%TCIMACRO{\dsum \limits_{i=1}^{m}}%
%BeginExpansion
{\displaystyle\sum\limits_{i=1}^{m}}
%EndExpansion
b_{i}f(y_{i})  &  \leq%
%TCIMACRO{\dsum \limits_{j=1}^{l}}%
%BeginExpansion
{\displaystyle\sum\limits_{j=1}^{l}}
%EndExpansion
a_{j}f(x_{j})-c\left(
%TCIMACRO{\dsum \limits_{j=1}^{l}}%
%BeginExpansion
{\displaystyle\sum\limits_{j=1}^{l}}
%EndExpansion
a_{j}x_{j}^{2}-%
%TCIMACRO{\dsum \limits_{i=1}^{m}}%
%BeginExpansion
{\displaystyle\sum\limits_{i=1}^{m}}
%EndExpansion
b_{i}y_{i}^{2}\right) \label{Sh_sr1}\\
&  \leq%
%TCIMACRO{\dsum \limits_{j=1}^{l}}%
%BeginExpansion
{\displaystyle\sum\limits_{j=1}^{l}}
%EndExpansion
a_{j}f(x_{j}).\nonumber
\end{align}
a) Specially, for $m=1$ and $\mathbf{b}=(1),$ (\ref{Sh_sr1}) becomes%
\begin{align*}
f\left(
%TCIMACRO{\dsum \limits_{j=1}^{l}}%
%BeginExpansion
{\displaystyle\sum\limits_{j=1}^{l}}
%EndExpansion
a_{j}x_{j}\right)   &  \leq%
%TCIMACRO{\dsum \limits_{j=1}^{l}}%
%BeginExpansion
{\displaystyle\sum\limits_{j=1}^{l}}
%EndExpansion
a_{j}f(x_{j})-c\left(
%TCIMACRO{\dsum \limits_{j=1}^{l}}%
%BeginExpansion
{\displaystyle\sum\limits_{j=1}^{l}}
%EndExpansion
a_{j}x_{j}^{2}-\bar{x}^{2}\right) \\
&  =%
%TCIMACRO{\dsum \limits_{j=1}^{l}}%
%BeginExpansion
{\displaystyle\sum\limits_{j=1}^{l}}
%EndExpansion
a_{j}f(x_{j})-c%
%TCIMACRO{\dsum \limits_{j=1}^{l}}%
%BeginExpansion
{\displaystyle\sum\limits_{j=1}^{l}}
%EndExpansion
a_{j}\left(  x_{j}-\bar{x}\right)  ^{2}\\
&  \leq%
%TCIMACRO{\dsum \limits_{j=1}^{l}}%
%BeginExpansion
{\displaystyle\sum\limits_{j=1}^{l}}
%EndExpansion
a_{j}f(x_{j}),
\end{align*}
where $\bar{x}=%
%TCIMACRO{\tsum \nolimits_{j=1}^{l}}%
%BeginExpansion
{\textstyle\sum\nolimits_{j=1}^{l}}
%EndExpansion
a_{j}x_{j},$ i.e. we get Jensen's inequality (\ref{Jen_strong}) for strongly
convex function.\newline b) For $m=l$ and $\mathbf{b}=\mathbf{e}=(1,...,1),$
(\ref{Sh_sr1}) becomes%
\begin{align*}%
%TCIMACRO{\dsum \limits_{i=1}^{m}}%
%BeginExpansion
{\displaystyle\sum\limits_{i=1}^{m}}
%EndExpansion
f(y_{i})  &  \leq%
%TCIMACRO{\dsum \limits_{i=1}^{m}}%
%BeginExpansion
{\displaystyle\sum\limits_{i=1}^{m}}
%EndExpansion
f(x_{i})-c\left(
%TCIMACRO{\dsum \limits_{i=1}^{m}}%
%BeginExpansion
{\displaystyle\sum\limits_{i=1}^{m}}
%EndExpansion
x_{i}^{2}-%
%TCIMACRO{\dsum \limits_{i=1}^{m}}%
%BeginExpansion
{\displaystyle\sum\limits_{i=1}^{m}}
%EndExpansion
y_{i}^{2}\right) \\
&  \leq%
%TCIMACRO{\dsum \limits_{i=1}^{m}}%
%BeginExpansion
{\displaystyle\sum\limits_{i=1}^{m}}
%EndExpansion
f(x_{i}),
\end{align*}
i.e. we get majorization inequality for strongly convex function.\newline c)
When $m=l,$ and all weights $b_{i}$ and $a_{j}$ are equal, then (\ref{Sh_str})
becomes%
\begin{align*}%
%TCIMACRO{\dsum \limits_{i=1}^{m}}%
%BeginExpansion
{\displaystyle\sum\limits_{i=1}^{m}}
%EndExpansion
a_{i}f(y_{i})  &  \leq%
%TCIMACRO{\dsum \limits_{i=1}^{m}}%
%BeginExpansion
{\displaystyle\sum\limits_{i=1}^{m}}
%EndExpansion
a_{i}f(x_{i})-c\left(
%TCIMACRO{\dsum \limits_{i=1}^{m}}%
%BeginExpansion
{\displaystyle\sum\limits_{i=1}^{m}}
%EndExpansion
a_{i}x_{i}^{2}-%
%TCIMACRO{\dsum \limits_{i=1}^{m}}%
%BeginExpansion
{\displaystyle\sum\limits_{i=1}^{m}}
%EndExpansion
a_{i}y_{i}^{2}\right) \\
&  \leq%
%TCIMACRO{\dsum \limits_{i=1}^{m}}%
%BeginExpansion
{\displaystyle\sum\limits_{i=1}^{m}}
%EndExpansion
a_{i}f(x_{i}),
\end{align*}
i.e. we get Fuchs' inequality for strongly convex function.
\end{remark}

Next we give conversion to Sherman's inequality for strongly convex functions.

\begin{theorem}
\label{Th_Con_Sherman_strong}Let $\mathbf{x}=(x_{1},...,x_{l})\in\lbrack
\alpha,\beta]^{l},$ $\mathbf{y}=(y_{1},...,y_{m})\in\lbrack\alpha,\beta]^{m},$
$\mathbf{a}=(a_{1},...,a_{l})\in\lbrack0,\infty)^{l}$ and $\mathbf{b}%
=(b_{1},...,b_{m})\in\lbrack0,\infty)^{m}$ be such that $(\mathbf{y,b}%
)\prec(\mathbf{x,a}).$ Let $B_{m}=\sum_{i=1}^{m}b_{i}.$ If $f:[\alpha
,\beta]\rightarrow\mathbb{R}$ is strongly convex with modulus $c>0,$ then%
\begin{align}
\sum_{j=1}^{l}a_{j}f\left(  x_{j}\right)   &  \leq\frac{B_{m}\beta-\sum
_{j=1}^{l}a_{j}x_{j}}{\beta-\alpha}f(\alpha)+\frac{\sum_{j=1}^{l}a_{j}%
x_{j}-B_{m}\alpha}{\beta-\alpha}f(\beta)\label{Conv_Sh}\\
&  -c%
%TCIMACRO{\dsum \limits_{j=1}^{l}}%
%BeginExpansion
{\displaystyle\sum\limits_{j=1}^{l}}
%EndExpansion
a_{j}(\beta-x_{j})(x_{j}-\alpha).\nonumber
\end{align}

\begin{proof}
Using (\ref{sh_assump_m}) we have%
\begin{equation}
\sum_{j=1}^{l}a_{j}f\left(  x_{j}\right)  =\sum_{j=1}^{l}\left(
%TCIMACRO{\dsum \limits_{i=1}^{m}}%
%BeginExpansion
{\displaystyle\sum\limits_{i=1}^{m}}
%EndExpansion
b_{j}a_{ij}\right)  f\left(  x_{j}\right)  =%
%TCIMACRO{\dsum \limits_{i=1}^{m}}%
%BeginExpansion
{\displaystyle\sum\limits_{i=1}^{m}}
%EndExpansion
b_{j}\left(  \sum_{j=1}^{l}a_{ij}f\left(  x_{j}\right)  \right)  .
\label{Conv1}%
\end{equation}
Applying (\ref{LR_str}) we get%
\begin{align}
\sum_{j=1}^{l}a_{ij}f\left(  x_{j}\right)   &  \leq\frac{\beta-\sum_{j=1}%
^{l}a_{ij}x_{j}}{\beta-\alpha}f(\alpha)+\frac{\sum_{j=1}^{l}a_{ij}x_{j}%
-\alpha}{\beta-\alpha}f(\beta)\label{Conv2}\\
&  -c%
%TCIMACRO{\dsum \limits_{j=1}^{l}}%
%BeginExpansion
{\displaystyle\sum\limits_{j=1}^{l}}
%EndExpansion
a_{ij}(\beta-x_{j})(x_{j}-\alpha).\nonumber
\end{align}
Now, combining (\ref{Conv1}) and (\ref{Conv2}), we get (\ref{Conv_Sh}).
\end{proof}
\end{theorem}

\begin{remark}
a)\ Specially, if $m=1$ and $\mathbf{b}=(1),$ then (\ref{Sh_sr1}) and
(\ref{Conv_Sh}) gives the following series of inequalities%
\begin{align*}
f\left(
%TCIMACRO{\dsum \limits_{j=1}^{l}}%
%BeginExpansion
{\displaystyle\sum\limits_{j=1}^{l}}
%EndExpansion
a_{j}x_{j}\right)   &  \leq%
%TCIMACRO{\dsum \limits_{j=1}^{l}}%
%BeginExpansion
{\displaystyle\sum\limits_{j=1}^{l}}
%EndExpansion
a_{j}f(x_{j})-c%
%TCIMACRO{\dsum \limits_{j=1}^{l}}%
%BeginExpansion
{\displaystyle\sum\limits_{j=1}^{l}}
%EndExpansion
a_{j}\left(  x_{j}-\bar{x}\right)  ^{2}\\
&  \leq%
%TCIMACRO{\dsum \limits_{j=1}^{l}}%
%BeginExpansion
{\displaystyle\sum\limits_{j=1}^{l}}
%EndExpansion
a_{j}f(x_{j})\\
&  \leq\frac{\beta-\sum_{j=1}^{l}a_{j}x_{j}}{\beta-\alpha}f(\alpha)+\frac
{\sum_{j=1}^{l}a_{j}x_{j}-\alpha}{\beta-\alpha}f(\beta)\\
&  -c%
%TCIMACRO{\dsum \limits_{j=1}^{l}}%
%BeginExpansion
{\displaystyle\sum\limits_{j=1}^{l}}
%EndExpansion
a_{j}(\beta-x_{j})(x_{j}-\alpha),
\end{align*}
i.e. we get Jensen's inequality and its conversion for strongly convex
functions.\newline b) If $m=l$ and $\mathbf{b}=\mathbf{e}=(1,...,1),$ then
(\ref{Sh_sr1}) and (\ref{Conv_Sh}) gives
\begin{align*}%
%TCIMACRO{\dsum \limits_{i=1}^{m}}%
%BeginExpansion
{\displaystyle\sum\limits_{i=1}^{m}}
%EndExpansion
f(y_{i})  &  \leq%
%TCIMACRO{\dsum \limits_{i=1}^{m}}%
%BeginExpansion
{\displaystyle\sum\limits_{i=1}^{m}}
%EndExpansion
f(x_{i})-c\left(
%TCIMACRO{\dsum \limits_{i=1}^{m}}%
%BeginExpansion
{\displaystyle\sum\limits_{i=1}^{m}}
%EndExpansion
x_{i}^{2}-%
%TCIMACRO{\dsum \limits_{i=1}^{m}}%
%BeginExpansion
{\displaystyle\sum\limits_{i=1}^{m}}
%EndExpansion
y_{i}^{2}\right) \\
&  \leq%
%TCIMACRO{\dsum \limits_{i=1}^{m}}%
%BeginExpansion
{\displaystyle\sum\limits_{i=1}^{m}}
%EndExpansion
f(x_{i})\\
&  \leq\frac{\beta-\sum_{j=1}^{l}x_{j}}{\beta-\alpha}f(\alpha)+\frac
{\sum_{j=1}^{l}x_{j}-\alpha}{\beta-\alpha}f(\beta)\\
&  -c%
%TCIMACRO{\dsum \limits_{j=1}^{l}}%
%BeginExpansion
{\displaystyle\sum\limits_{j=1}^{l}}
%EndExpansion
(\beta-x_{j})(x_{j}-\alpha),
\end{align*}
i.e. we get majorization inequality and its conversion for strongly convex
functions.\newline c) If $m=l,$ and all weights $b_{i}$ and $a_{j}$ are equal,
then (\ref{Sh_sr1}) and (\ref{Conv_Sh}) gives%
\begin{align*}%
%TCIMACRO{\dsum \limits_{i=1}^{m}}%
%BeginExpansion
{\displaystyle\sum\limits_{i=1}^{m}}
%EndExpansion
a_{i}f(y_{i})  &  \leq%
%TCIMACRO{\dsum \limits_{i=1}^{m}}%
%BeginExpansion
{\displaystyle\sum\limits_{i=1}^{m}}
%EndExpansion
a_{i}f(x_{i})-c\left(
%TCIMACRO{\dsum \limits_{i=1}^{m}}%
%BeginExpansion
{\displaystyle\sum\limits_{i=1}^{m}}
%EndExpansion
a_{i}x_{i}^{2}-%
%TCIMACRO{\dsum \limits_{i=1}^{m}}%
%BeginExpansion
{\displaystyle\sum\limits_{i=1}^{m}}
%EndExpansion
a_{i}y_{i}^{2}\right) \\
&  \leq%
%TCIMACRO{\dsum \limits_{i=1}^{m}}%
%BeginExpansion
{\displaystyle\sum\limits_{i=1}^{m}}
%EndExpansion
a_{i}f(x_{i})\\
&  \leq\frac{A_{m}\beta-\sum_{i=1}^{m}a_{i}x_{i}}{\beta-\alpha}f(\alpha
)+\frac{\sum_{i=1}^{m}a_{i}x_{i}-A_{m}\alpha}{\beta-\alpha}f(\beta)\\
&  -c%
%TCIMACRO{\dsum \limits_{i=1}^{m}}%
%BeginExpansion
{\displaystyle\sum\limits_{i=1}^{m}}
%EndExpansion
a_{i}(\beta-x_{i})(x_{i}-\alpha),
\end{align*}
where $\sum_{i=1}^{m}a_{i}=A_{m},$ i.e. we get Fuchs' inequality and its
conversion for strongly convex functions.
\end{remark}

\section{Applications to $f$-divergences}

Shannon \cite{Shann} introduced a statistical concept of entropy in the theory
of communication and transmission of information, the measure of information
defined by
\begin{equation}
H(\mathbf{\mathbf{p}})=\sum\limits_{i=1}^{n}p_{i}\ln\frac{1}{p_{i}},
\label{Sh}%
\end{equation}
where $\mathbf{p}=(p_{1},...,p_{n})$ is a positive probability distribution ,
i.e. $p_{i}>0,$ $i=1,...,n$, with $%
%TCIMACRO{\tsum \nolimits_{i=1}^{n}}%
%BeginExpansion
{\textstyle\sum\nolimits_{i=1}^{n}}
%EndExpansion
p_{i}=1,$ for some discrete random variable $X.$ It satisfied estimate%
\[
0\leqslant H(\mathbf{\mathbf{p}})\leqslant\ln n.
\]
Shannon's entropy quantifies the unevenness in the probability distribution
$\mathbf{\mathbf{p}}$.

As a slight modification of the previous formula, we get the Kullback-Leibler
divergence \cite{Kull} or relative entropy of $\mathbf{q}$ with respect to
$\mathbf{p}$ defined by
\[
KL(\mathbf{\mathbf{q},\mathbf{p}})=%
%TCIMACRO{\dsum \limits_{i=1}^{n}}%
%BeginExpansion
{\displaystyle\sum\limits_{i=1}^{n}}
%EndExpansion
q_{i}\left(  \ln q_{i}-\ln p_{i}\right)  =%
%TCIMACRO{\dsum \limits_{i=1}^{n}}%
%BeginExpansion
{\displaystyle\sum\limits_{i=1}^{n}}
%EndExpansion
q_{i}\ln\left(  \frac{q_{i}}{p_{i}}\right)  .
\]
It is a measure of the difference between two positive probability
distributions $\mathbf{\mathbf{q}}$ and $\mathbf{\mathbf{p}}$ over the same
variable. In statistics, it arises as the expected logarithm of difference
between the probability $\mathbf{q}$ of data in the original distribution with
the approximating distribution $\mathbf{p}$. It satisfies the following
estimates%
\[
KL(\mathbf{\mathbf{q},\mathbf{p}})\geq0.
\]

The previous two concepts we can get as special cases of the Csisz\'{a}r
$f$-divergence functional
\begin{equation}
D_{f}(\mathbf{q},\mathbf{p})=%
%TCIMACRO{\dsum \limits_{i=1}^{n}}%
%BeginExpansion
{\displaystyle\sum\limits_{i=1}^{n}}
%EndExpansion
p_{i}f\left(  \frac{q_{i}}{p_{i}}\right)  , \label{Cs_f_div1}%
\end{equation}
where $f:(0,\infty)\rightarrow\mathbb{R}$ is a convex function and
$\mathbf{p}=(p_{1},...,p_{n}),$ $\mathbf{q}=(q_{1},...,q_{n})$ with
$p_{i},q_{i}>0,$ $i=1,...,n$ (see \cite{CSIS}, \cite{CK})$.$

Note that%
\begin{align*}
H(\mathbf{p})  &  =-\sum_{i=1}^{n}{p_{i}}\ln p_{i}=-D_{f}(\mathbf{e}%
,\mathbf{p}),\text{ \ \ }f(t)=-\ln t,\\
D_{KL}(\mathbf{q,p})  &  =%
%TCIMACRO{\dsum \limits_{i=1}^{m}}%
%BeginExpansion
{\displaystyle\sum\limits_{i=1}^{m}}
%EndExpansion
q_{i}\ln\frac{q_{i}}{p_{i}}=D_{f}(\mathbf{q},\mathbf{p}),\text{ \ \ }f(t)=t\ln
t.
\end{align*}

Csisz\'{a}r with K\"{o}rner \cite{CK} proved Jensen's inequality for the $f$
-divergence functional as follows%
\begin{equation}%
%TCIMACRO{\dsum \limits_{i=1}^{n}}%
%BeginExpansion
{\displaystyle\sum\limits_{i=1}^{n}}
%EndExpansion
q_{i}f\left(  \frac{%
%TCIMACRO{\dsum \limits_{i=1}^{n}}%
%BeginExpansion
{\displaystyle\sum\limits_{i=1}^{n}}
%EndExpansion
p_{i}}{%
%TCIMACRO{\dsum \limits_{i=1}^{n}}%
%BeginExpansion
{\displaystyle\sum\limits_{i=1}^{n}}
%EndExpansion
q_{i}}\right)  \leqslant D_{f}(\mathbf{q},\mathbf{p}). \tag{CK}%
\end{equation}
Specially, if $f$ is normalized, i.e. $f(1)=0$ and $%
%TCIMACRO{\tsum \nolimits_{i=1}^{n}}%
%BeginExpansion
{\textstyle\sum\nolimits_{i=1}^{n}}
%EndExpansion
p_{i}=%
%TCIMACRO{\tsum \nolimits_{i=1}^{n}}%
%BeginExpansion
{\textstyle\sum\nolimits_{i=1}^{n}}
%EndExpansion
q_{i},$ then%
\begin{equation}
0\leq D_{f}(\mathbf{q},\mathbf{p}). \label{CK_l}%
\end{equation}

Csisz\'{a}r $f$-divergence functional (\ref{Cs_f_div1}) is widely employed in
different scientic fields among which we point out mathematical statistics and
specially information theory with deep connections in topics as diverse as
artificial intelligence, statistical physics, and biological evolution. For
suitable choices of the kernel $f,$ the general aspect of the Csisz\'{a}r
$f$-divergence functional (\ref{Cs_f_div1}) can be interpreted as a series of
the well-known divergencies (see \cite{Drag}, \cite{Kap}, \cite{Kap1})$.$ Here
we give some examples:

\begin{itemize}
\item Hellinger divergence
\[
h^{2}(\mathbf{q,p})=\frac{1}{2}\sum_{i=1}^{n}(\sqrt{p_{i}}-\sqrt{q_{i}}%
)^{2},\text{ \ \ }f(t)=\frac{1}{2}\left(  \sqrt{t}-1\right)  ^{2},
\]

\item Variational distance%
\[
V(\mathbf{q,p})=\sum_{i=1}^{n}|p_{i}-q_{i}|,\text{ \ \ }f(t)=|t-1|,
\]

\item Harmonic divergence
\[
D_{H}(\mathbf{q,p})=\sum_{i=1}^{n}\frac{2p_{i}q_{i}}{p_{i}+q_{i}},\text{
\ \ }f(t)=\frac{2t}{1+t},
\]

\item Bhattacharya distance%
\[
D_{B}(\mathbf{q,p})=-D_{f}(\mathbf{q,p})=%
%TCIMACRO{\dsum \limits_{i=1}^{n}}%
%BeginExpansion
{\displaystyle\sum\limits_{i=1}^{n}}
%EndExpansion
\sqrt{p_{i}q_{i}},\text{ \ \ }f(t)=-\sqrt{t},
\]

\item Triangular discrimination%
\[
D_{T}(\mathbf{q,p})=\sum_{i=1}^{n}\frac{(p_{i}-q_{i})^{2}}{p_{i}+q_{i}},\text{
\ \ }f(t)=\frac{(t-1)^{2}}{t+1},
\]

\item Chi square distance%
\[
D_{\chi^{2}}(\mathbf{q,p})=%
%TCIMACRO{\dsum \limits_{i=1}^{n}}%
%BeginExpansion
{\displaystyle\sum\limits_{i=1}^{n}}
%EndExpansion
\frac{\left(  q_{i}-p_{i}\right)  ^{2}}{p_{i}},\text{ \ \ }f(t)=(t-1)^{2},
\]

\item R\'{e}nyi $\alpha$-order entropy ($\alpha>1$)%
\[
R_{\alpha}(\mathbf{q,p})=%
%TCIMACRO{\dsum \limits_{i=1}^{n}}%
%BeginExpansion
{\displaystyle\sum\limits_{i=1}^{n}}
%EndExpansion
q_{i}^{\alpha}p_{i}^{1-\alpha},\ \ f(t)=t^{\alpha}.
\]

\end{itemize}

We extend definition of $f$-divergence functional (\ref{Cs_f_div1}) as follows.

\begin{definition}
Let $f:(0,\infty)\rightarrow\mathbb{R}$ be a strongly convex function with
modulus $c>0$ and $\mathbf{p}=(p_{1},...,p_{n}),$ $\mathbf{q}=(q_{1}%
,...,q_{n})$ with $p_{i},q_{i}>0,$ $i=1,...,n.$ We define%
\begin{equation}
\tilde{D}_{f}(\mathbf{q},\mathbf{p})=%
%TCIMACRO{\dsum \limits_{i=1}^{n}}%
%BeginExpansion
{\displaystyle\sum\limits_{i=1}^{n}}
%EndExpansion
p_{i}f\left(  \frac{q_{i}}{p_{i}}\right)  . \label{Cs_f_gen}%
\end{equation}

\end{definition}

In this section our intention is to derive mutual bounds for the generalized
$f$-divergence functional (\ref{Cs_f_gen}). We obtain some reverse relations
for the generalized $f$-divergence functional that correspond to the class of
strongly convex functions.

Through the rest of the paper we always assume that $\alpha,\beta>0.$

\begin{corollary}
Let $\mathbf{p}=(p_{1},...,p_{l})\in\lbrack\alpha,\beta]^{l},$ $\mathbf{q}%
=(q_{1},...,q_{l})\in\lbrack\alpha,\beta]^{l}$ and $\mathbf{R}=(r_{ij}%
)\in\mathcal{M}_{ml}(\mathbb{R})$ be column stochastic matrix. Let us define
$\left\langle \mathbf{p},\mathbf{r}_{i}\right\rangle =%
%TCIMACRO{\tsum \nolimits_{j=1}^{l}}%
%BeginExpansion
{\textstyle\sum\nolimits_{j=1}^{l}}
%EndExpansion
p_{j}r_{ij}>0,$\ $\left\langle \mathbf{q},\mathbf{r}_{i}\right\rangle =%
%TCIMACRO{\tsum \nolimits_{j=1}^{l}}%
%BeginExpansion
{\textstyle\sum\nolimits_{j=1}^{l}}
%EndExpansion
q_{j}r_{ij},$\ $i=1,...,m.$ Then for every $f:[\alpha,\beta]\rightarrow
\mathbb{R}$ strongly convex with modulus $c>0,$ we have%
\begin{align}
\sum_{i=1}^{m}\left\langle \mathbf{p},\mathbf{r}_{i}\right\rangle f\left(
\frac{\left\langle \mathbf{q},\mathbf{r}_{i}\right\rangle }{\left\langle
\mathbf{p},\mathbf{r}_{i}\right\rangle }\right)   &  \leq\tilde{D}%
_{f}(\mathbf{q},\mathbf{p})-c\left(
%TCIMACRO{\dsum \limits_{j=1}^{l}}%
%BeginExpansion
{\displaystyle\sum\limits_{j=1}^{l}}
%EndExpansion
\frac{q_{j}}{p_{j}}^{2}-%
%TCIMACRO{\dsum \limits_{i=1}^{m}}%
%BeginExpansion
{\displaystyle\sum\limits_{i=1}^{m}}
%EndExpansion
\frac{\left\langle \mathbf{q},\mathbf{r}_{i}\right\rangle }{\left\langle
\mathbf{p},\mathbf{r}_{i}\right\rangle }^{2}\right) \label{Sh_div1}\\
&  \leq\tilde{D}_{f}(\mathbf{q},\mathbf{p})\nonumber\\
&  \leq\frac{\sum\limits_{i=1}^{m}\left\langle \mathbf{p},\mathbf{r}%
_{i}\right\rangle \beta-\sum\limits_{j=1}^{l}q_{j}}{\beta-\alpha}%
f(\alpha)+\frac{\sum\limits_{j=1}^{l}q_{j}-\sum\limits_{i=1}^{m}\left\langle
\mathbf{p},\mathbf{r}_{i}\right\rangle \alpha}{\beta-\alpha}f(\beta
)\nonumber\\
&  -c%
%TCIMACRO{\dsum \limits_{j=1}^{l}}%
%BeginExpansion
{\displaystyle\sum\limits_{j=1}^{l}}
%EndExpansion
p_{j}\left(  \beta-\frac{q_{j}}{p_{j}}\right)  \left(  \frac{q_{j}}{p_{j}%
}-\alpha\right)  .\nonumber
\end{align}

\end{corollary}

\begin{proof}
Let us consider $\mathbf{x}=(x_{1},...,x_{l})$ and $\mathbf{y}=(y_{1}%
,...,y_{m}),$ such that $x_{j}=\frac{q_{j}}{p_{j}},$ $j=1,...,l$ and
$y_{i}=\frac{\left\langle \mathbf{q},\mathbf{r}_{i}\right\rangle
}{\left\langle \mathbf{p},\mathbf{r}_{i}\right\rangle },$ $i=1,...,m.$ Let
$a_{j}=%
%TCIMACRO{\tsum \limits_{i=1}^{m}}%
%BeginExpansion
{\textstyle\sum\limits_{i=1}^{m}}
%EndExpansion
b_{i}\frac{p_{j}r_{ij}}{\left\langle \mathbf{p},\mathbf{r}_{i}\right\rangle
},j=1,...,m,$ where $b_{i}=\left\langle \mathbf{p},\mathbf{r}_{i}\right\rangle
,$ $i=1,...,m.$ \newline We have%
\[
\frac{\left\langle \mathbf{q},\mathbf{r}_{i}\right\rangle }{\left\langle
\mathbf{p},\mathbf{r}_{i}\right\rangle }=\frac{%
%TCIMACRO{\tsum \limits_{j=1}^{l}}%
%BeginExpansion
{\textstyle\sum\limits_{j=1}^{l}}
%EndExpansion
q_{j}r_{ij}}{%
%TCIMACRO{\tsum \limits_{j=1}^{l}}%
%BeginExpansion
{\textstyle\sum\limits_{j=1}^{l}}
%EndExpansion
p_{j}r_{ij}}=\frac{p_{1}r_{i1}}{%
%TCIMACRO{\tsum \limits_{j=1}^{l}}%
%BeginExpansion
{\textstyle\sum\limits_{j=1}^{l}}
%EndExpansion
p_{j}r_{ij}}\frac{q_{1}}{p_{1}}+...+\frac{p_{l}r_{il}}{%
%TCIMACRO{\tsum \limits_{j=1}^{l}}%
%BeginExpansion
{\textstyle\sum\limits_{j=1}^{l}}
%EndExpansion
p_{j}r_{ij}}\frac{q_{l}}{p_{l}},\text{ \ \ }i=1,...,m.
\]
Moreover, the following identity%
\[
\left(  \frac{\left\langle \mathbf{q},\mathbf{r}_{1}\right\rangle
}{\left\langle \mathbf{p},\mathbf{r}_{1}\right\rangle },...,\frac{\left\langle
\mathbf{q},\mathbf{r}_{m}\right\rangle }{\left\langle \mathbf{p}%
,\mathbf{r}_{m}\right\rangle }\right)  =\left(  \frac{q_{1}}{p_{1}}%
,...,\frac{q_{l}}{p_{l}}\right)  \cdot\left(
\begin{array}
[c]{ccc}%
\frac{p_{1}r_{11}}{\left\langle \mathbf{p},\mathbf{r}_{1}\right\rangle } &
\cdots & \frac{p_{1}r_{m1}}{\left\langle \mathbf{p},\mathbf{r}_{m}%
\right\rangle }\\
\vdots & \ddots & \vdots\\
\frac{p_{l}r_{1l}}{\left\langle \mathbf{p},\mathbf{r}_{1}\right\rangle } &
\cdots & \frac{p_{l}r_{ml}}{\left\langle \mathbf{p},\mathbf{r}_{m}%
\right\rangle }%
\end{array}
\right)
\]
holds for some row stochastic matrix $\mathbf{A}=(a_{ij})\in\mathcal{M}%
_{ml}(\mathbb{R}),$ with $a_{ij}=\frac{p_{j}r_{ij}}{\left\langle
\mathbf{p},\mathbf{r}_{i}\right\rangle },$ $i=1,...,m,$ $j=1,...,l.$
Therefore, $\mathbf{y}=\mathbf{xA}^{T}$ holds$.$ \newline Further, we have
\[
a_{j}=%
%TCIMACRO{\dsum \limits_{i=1}^{m}}%
%BeginExpansion
{\displaystyle\sum\limits_{i=1}^{m}}
%EndExpansion
\left\langle \mathbf{p},\mathbf{r}_{i}\right\rangle \frac{p_{j}r_{ij}%
}{\left\langle \mathbf{p},\mathbf{r}_{i}\right\rangle }=p_{j}%
%TCIMACRO{\dsum \limits_{i=1}^{m}}%
%BeginExpansion
{\displaystyle\sum\limits_{i=1}^{m}}
%EndExpansion
r_{ij}=p_{j},j=1,...,l,
\]
i.e. $\mathbf{a}=\mathbf{bA}.$ Therefore, the assumptions of Theorem
\ref{Th_Sherman_strong} and Theorem \ref{Th_Con_Sherman_strong} are fulfill.
Now applying (\ref{Sh_str}) and (\ref{Conv_Sh}), we get
\begin{align*}
\sum_{i=1}^{m}\left\langle \mathbf{p},\mathbf{r}_{i}\right\rangle f\left(
\frac{\left\langle \mathbf{q},\mathbf{r}_{i}\right\rangle }{\left\langle
\mathbf{p},\mathbf{r}_{i}\right\rangle }\right)   &  \leq\sum_{j=1}^{l}%
p_{j}f\left(  \frac{q_{j}}{p_{j}}\right)  -c\left(
%TCIMACRO{\dsum \limits_{j=1}^{l}}%
%BeginExpansion
{\displaystyle\sum\limits_{j=1}^{l}}
%EndExpansion
\frac{q_{j}}{p_{j}}^{2}-%
%TCIMACRO{\dsum \limits_{i=1}^{m}}%
%BeginExpansion
{\displaystyle\sum\limits_{i=1}^{m}}
%EndExpansion
\frac{\left\langle \mathbf{q},\mathbf{r}_{i}\right\rangle }{\left\langle
\mathbf{p},\mathbf{r}_{i}\right\rangle }^{2}\right) \\
&  \leq\sum_{j=1}^{l}p_{j}f\left(  \frac{q_{j}}{p_{j}}\right) \\
&  \leq\frac{\sum\limits_{i=1}^{m}\left\langle \mathbf{p},\mathbf{r}%
_{i}\right\rangle \beta-\sum\limits_{j=1}^{l}q_{j}}{\beta-\alpha}%
f(\alpha)+\frac{\sum\limits_{j=1}^{l}q_{j}-\sum\limits_{i=1}^{m}\left\langle
\mathbf{p},\mathbf{r}_{i}\right\rangle \alpha}{\beta-\alpha}f(\beta)\\
&  -c%
%TCIMACRO{\dsum \limits_{j=1}^{l}}%
%BeginExpansion
{\displaystyle\sum\limits_{j=1}^{l}}
%EndExpansion
p_{j}\left(  \beta-\frac{q_{j}}{p_{j}}\right)  \left(  \frac{q_{j}}{p_{j}%
}-\alpha\right)
\end{align*}
what is equivalent to (\ref{Sh_div1}).
\end{proof}

Specially, for $m=1,$ the previous result reduces to the next corollary.

\begin{corollary}
Let $\mathbf{p}=(p_{1},...,p_{l})\in\lbrack\alpha,\beta]^{l},$ $\mathbf{q}%
=(q_{1},...,q_{l})\in\lbrack\alpha,\beta]^{l}$ and and $\mathbf{r}%
=(r_{1},...,r_{l})\in\lbrack\alpha,\beta]^{l}$. Let us define $\left\langle
\mathbf{p},\mathbf{r}\right\rangle =%
%TCIMACRO{\tsum \nolimits_{j=1}^{l}}%
%BeginExpansion
{\textstyle\sum\nolimits_{j=1}^{l}}
%EndExpansion
p_{j}r_{j}>0,$\ $\left\langle \mathbf{q},\mathbf{r}\right\rangle =%
%TCIMACRO{\tsum \nolimits_{j=1}^{l}}%
%BeginExpansion
{\textstyle\sum\nolimits_{j=1}^{l}}
%EndExpansion
q_{j}r_{j}.$ Then for every $f:[\alpha,\beta]\rightarrow\mathbb{R}$ strongly
convex with modulus $c>0,$ we have%
\begin{align*}
\left\langle \mathbf{p},\mathbf{r}\right\rangle f\left(  \frac{\left\langle
\mathbf{q},\mathbf{r}\right\rangle }{\left\langle \mathbf{p},\mathbf{r}%
\right\rangle }\right)   &  \leq\tilde{D}_{f}(\mathbf{q},\mathbf{p})-c\left(
%TCIMACRO{\dsum \limits_{j=1}^{l}}%
%BeginExpansion
{\displaystyle\sum\limits_{j=1}^{l}}
%EndExpansion
\frac{q_{j}}{p_{j}}^{2}-\frac{\left\langle \mathbf{q},\mathbf{r}\right\rangle
}{\left\langle \mathbf{p},\mathbf{r}\right\rangle }^{2}\right) \\
&  \leq\tilde{D}_{f}(\mathbf{q},\mathbf{p})\\
&  \leq\frac{\left\langle \mathbf{p},\mathbf{r}\right\rangle \beta
-\sum\limits_{j=1}^{l}q_{j}}{\beta-\alpha}f(\alpha)+\frac{\sum\limits_{j=1}%
^{l}q_{j}-\left\langle \mathbf{p},\mathbf{r}\right\rangle \alpha}{\beta
-\alpha}f(\beta)\\
&  -c%
%TCIMACRO{\dsum \limits_{j=1}^{l}}%
%BeginExpansion
{\displaystyle\sum\limits_{j=1}^{l}}
%EndExpansion
p_{j}\left(  \beta-\frac{q_{j}}{p_{j}}\right)  \left(  \frac{q_{j}}{p_{j}%
}-\alpha\right)  .
\end{align*}
If in addition $\mathbf{r}=\mathbf{e}=(1,...,1)$, then%
\begin{align*}%
%TCIMACRO{\dsum \limits_{j=1}^{l}}%
%BeginExpansion
{\displaystyle\sum\limits_{j=1}^{l}}
%EndExpansion
p_{j}f\left(  \frac{%
%TCIMACRO{\dsum \limits_{j=1}^{l}}%
%BeginExpansion
{\displaystyle\sum\limits_{j=1}^{l}}
%EndExpansion
q_{j}}{%
%TCIMACRO{\dsum \limits_{j=1}^{l}}%
%BeginExpansion
{\displaystyle\sum\limits_{j=1}^{l}}
%EndExpansion
p_{j}}\right)   &  \leq\tilde{D}_{f}(\mathbf{q},\mathbf{p})-c\left(
%TCIMACRO{\dsum \limits_{j=1}^{l}}%
%BeginExpansion
{\displaystyle\sum\limits_{j=1}^{l}}
%EndExpansion
\frac{q_{j}}{p_{j}}^{2}-\frac{\left(
%TCIMACRO{\dsum \limits_{j=1}^{l}}%
%BeginExpansion
{\displaystyle\sum\limits_{j=1}^{l}}
%EndExpansion
q_{j}\right)  }{%
%TCIMACRO{\dsum \limits_{j=1}^{l}}%
%BeginExpansion
{\displaystyle\sum\limits_{j=1}^{l}}
%EndExpansion
p_{j}}^{2}\right) \\
&  \leq\tilde{D}_{f}(\mathbf{q},\mathbf{p})\\
&  \leq\frac{%
%TCIMACRO{\tsum \limits_{j=1}^{l}}%
%BeginExpansion
{\textstyle\sum\limits_{j=1}^{l}}
%EndExpansion
p_{j}\beta-%
%TCIMACRO{\tsum \limits_{j=1}^{l}}%
%BeginExpansion
{\textstyle\sum\limits_{j=1}^{l}}
%EndExpansion
q_{j}}{\beta-\alpha}f(\alpha)+\frac{%
%TCIMACRO{\tsum \limits_{j=1}^{l}}%
%BeginExpansion
{\textstyle\sum\limits_{j=1}^{l}}
%EndExpansion
q_{j}-%
%TCIMACRO{\tsum \limits_{j=1}^{l}}%
%BeginExpansion
{\textstyle\sum\limits_{j=1}^{l}}
%EndExpansion
p_{j}\alpha}{\beta-\alpha}f(\beta)\\
&  -c%
%TCIMACRO{\dsum \limits_{j=1}^{l}}%
%BeginExpansion
{\displaystyle\sum\limits_{j=1}^{l}}
%EndExpansion
p_{j}\left(  \beta-\frac{q_{j}}{p_{j}}\right)  \left(  \frac{q_{j}}{p_{j}%
}-\alpha\right)  .
\end{align*}
Moreover, if $f$ is normalized, i.e. $f(1)=0$ and $%
%TCIMACRO{\tsum \nolimits_{j=1}^{l}}%
%BeginExpansion
{\textstyle\sum\nolimits_{j=1}^{l}}
%EndExpansion
p_{j}=%
%TCIMACRO{\tsum \nolimits_{j=1}^{l}}%
%BeginExpansion
{\textstyle\sum\nolimits_{j=1}^{l}}
%EndExpansion
q_{j},$ we get%
\begin{align*}
0  &  \leq\tilde{D}_{f}(\mathbf{q},\mathbf{p})-c\left(
%TCIMACRO{\dsum \limits_{j=1}^{l}}%
%BeginExpansion
{\displaystyle\sum\limits_{j=1}^{l}}
%EndExpansion
\frac{q_{j}}{p_{j}}^{2}-%
%TCIMACRO{\dsum \limits_{j=1}^{l}}%
%BeginExpansion
{\displaystyle\sum\limits_{j=1}^{l}}
%EndExpansion
p_{j}\right) \\
&  \leq\tilde{D}_{f}(\mathbf{q},\mathbf{p})\\
&  \leq\frac{%
%TCIMACRO{\tsum \limits_{j=1}^{l}}%
%BeginExpansion
{\textstyle\sum\limits_{j=1}^{l}}
%EndExpansion
p_{j}(\beta-1)}{\beta-\alpha}f(\alpha)+\frac{%
%TCIMACRO{\tsum \limits_{j=1}^{l}}%
%BeginExpansion
{\textstyle\sum\limits_{j=1}^{l}}
%EndExpansion
p_{j}(1-\alpha)}{\beta-\alpha}f(\beta)-c%
%TCIMACRO{\dsum \limits_{j=1}^{l}}%
%BeginExpansion
{\displaystyle\sum\limits_{j=1}^{l}}
%EndExpansion
p_{j}\left(  \beta-\frac{q_{j}}{p_{j}}\right)  \left(  \frac{q_{j}}{p_{j}%
}-\alpha\right)  .
\end{align*}

\end{corollary}

\section{Generalization of Sherman's inequality for strongly $n$-convex
function}

The technique that we use in this section is based on an application of Fink's
identity \cite{FINK}%
\begin{align}
f(x)  &  =\frac{n}{\beta-\alpha}%
%TCIMACRO{\dint \nolimits_{\alpha}^{\beta}}%
%BeginExpansion
{\displaystyle\int\nolimits_{\alpha}^{\beta}}
%EndExpansion
f(t)dt-\sum_{w=1}^{n-1}\frac{n-w}{w!}\cdot\frac{f^{(w-1)}(\alpha
)(x-\alpha)^{w}-f^{(w-1)}(\beta)(x-\beta)^{w}}{\beta-\alpha}\nonumber\\
&  +\frac{1}{(n-1)!(\beta-\alpha)}%
%TCIMACRO{\dint \nolimits_{\alpha}^{\beta}}%
%BeginExpansion
{\displaystyle\int\nolimits_{\alpha}^{\beta}}
%EndExpansion
(x-t)^{n-1}k(t,x)f^{(n)}(t)dt, \label{Fink_id0}%
\end{align}
where%
\begin{equation}
k(t,x)=\left\{
\begin{tabular}
[c]{ll}%
$t-\alpha,$ & $\alpha\leq t\leq x\leq\beta$\\
$t-\beta,$ & $\alpha\leq x<t\leq\beta$%
\end{tabular}
\ \ \ \ \ \ \ \ \ \ \ \ \right.  , \label{fink_id_w}%
\end{equation}
which holds for every $f:[\alpha,\beta]\rightarrow\mathbb{R}$ such that
$f^{(n-1)}$ is absolutely continuous for some $n\geq1.$ The sum in
(\ref{Fink_id0}) is zero when $n=1$.

We start with an identity which is very useful for us to obtain generalizations.

\begin{theorem}
\label{Th_Fink1}Let $\mathbf{x}=(x_{1},...,x_{l})\in\lbrack\alpha,\beta]^{l},$
$\mathbf{y}=(y_{1},...,y_{m})\in\lbrack\alpha,\beta]^{m},$ $\mathbf{a}%
=(a_{1},...,a_{l})\in\lbrack0,\infty)^{l}$ and $\mathbf{b}=(b_{1}%
,...,b_{m})\in\lbrack0,\infty)^{m}$ be such that $(\mathbf{y,b})\prec
(\mathbf{x,a}).$ Let $k(t,\cdot)$ be defined as in (\ref{fink_id_w}). Then for
every $f:[\alpha,\beta]\rightarrow\mathbb{R}$, such that $f^{(n-1)}$ is
absolutely continuous on $[\alpha,\beta],$ we have%
\begin{align}
&
%TCIMACRO{\dsum \limits_{j=1}^{l}}%
%BeginExpansion
{\displaystyle\sum\limits_{j=1}^{l}}
%EndExpansion
a_{j}f(x_{j})-%
%TCIMACRO{\dsum \limits_{i=1}^{m}}%
%BeginExpansion
{\displaystyle\sum\limits_{i=1}^{m}}
%EndExpansion
b_{i}f(y_{i})\label{Lm1}\\
&  =\frac{1}{\beta-\alpha}\sum_{w=2}^{n-1}\frac{n-w}{w!}\cdot f^{(w-1)}%
(\beta)\left(
%TCIMACRO{\dsum \limits_{j=1}^{l}}%
%BeginExpansion
{\displaystyle\sum\limits_{j=1}^{l}}
%EndExpansion
a_{j}(x_{j}-\beta)^{w}-%
%TCIMACRO{\dsum \limits_{i=1}^{m}}%
%BeginExpansion
{\displaystyle\sum\limits_{i=1}^{m}}
%EndExpansion
b_{i}(y_{i}-\beta)^{w}\right) \nonumber\\
&  -\frac{1}{\beta-\alpha}\sum_{w=2}^{n-1}\frac{n-w}{w!}\cdot f^{(w-1)}%
(\alpha)\left(
%TCIMACRO{\dsum \limits_{j=1}^{l}}%
%BeginExpansion
{\displaystyle\sum\limits_{j=1}^{l}}
%EndExpansion
a_{j}(x_{j}-\alpha)^{w}-%
%TCIMACRO{\dsum \limits_{i=1}^{m}}%
%BeginExpansion
{\displaystyle\sum\limits_{i=1}^{m}}
%EndExpansion
b_{i}(y_{i}-\alpha)^{w}\right) \nonumber\\
&  +\frac{1}{(n-1)!(\beta-\alpha)}\times\nonumber\\
&
%TCIMACRO{\dint \nolimits_{\alpha}^{\beta}}%
%BeginExpansion
{\displaystyle\int\nolimits_{\alpha}^{\beta}}
%EndExpansion
\left[
%TCIMACRO{\dsum \limits_{j=1}^{l}}%
%BeginExpansion
{\displaystyle\sum\limits_{j=1}^{l}}
%EndExpansion
a_{j}(x_{j}-t)^{n-1}k(t,x_{j})-%
%TCIMACRO{\dsum \limits_{i=1}^{m}}%
%BeginExpansion
{\displaystyle\sum\limits_{i=1}^{m}}
%EndExpansion
b_{i}(y_{i}-t)^{n-1}k(t,y_{i})\right]  f^{(n)}(t)dt.\nonumber
\end{align}

\begin{proof}
Applying (\ref{Fink_id0}) to the Sherman difference $%
%TCIMACRO{\tsum \nolimits_{j=1}^{l}}%
%BeginExpansion
{\textstyle\sum\nolimits_{j=1}^{l}}
%EndExpansion
a_{j}f(x_{j})-%
%TCIMACRO{\tsum \nolimits_{i=1}^{m}}%
%BeginExpansion
{\textstyle\sum\nolimits_{i=1}^{m}}
%EndExpansion
b_{i}f(y_{i}),$ we get (\ref{Lm1}).
\end{proof}
\end{theorem}

\begin{theorem}
\label{Th_Fink2a}Let all the assumptions of Theorem \ref{Th_Fink1} be
satisfied. Additionally, let $f$ be $n$-strongly convex with modulus $c>0.$ If%
\begin{equation}%
%TCIMACRO{\dsum \limits_{j=1}^{l}}%
%BeginExpansion
{\displaystyle\sum\limits_{j=1}^{l}}
%EndExpansion
a_{j}(x_{j}-t)^{n-1}k(t,x_{j})-%
%TCIMACRO{\dsum \limits_{i=1}^{m}}%
%BeginExpansion
{\displaystyle\sum\limits_{i=1}^{m}}
%EndExpansion
b_{i}(y_{i}-t)^{n-1}k(t,y_{i})\geq0,\text{ \ }\alpha\leq t\leq\beta,
\label{Th_Fink2_assump}%
\end{equation}
then%
\begin{align}
&
%TCIMACRO{\dsum \limits_{j=1}^{l}}%
%BeginExpansion
{\displaystyle\sum\limits_{j=1}^{l}}
%EndExpansion
a_{j}f(x_{j})-%
%TCIMACRO{\dsum \limits_{i=1}^{m}}%
%BeginExpansion
{\displaystyle\sum\limits_{i=1}^{m}}
%EndExpansion
b_{i}f(y_{i})-c\left(
%TCIMACRO{\dsum \limits_{j=1}^{l}}%
%BeginExpansion
{\displaystyle\sum\limits_{j=1}^{l}}
%EndExpansion
a_{j}x_{j}^{n}-%
%TCIMACRO{\dsum \limits_{i=1}^{m}}%
%BeginExpansion
{\displaystyle\sum\limits_{i=1}^{m}}
%EndExpansion
b_{i}y_{i}^{n}\right) \label{Th_Fink2a_rez}\\
&  \geq\frac{1}{\beta-\alpha}\sum_{w=1}^{n-1}\frac{n-w}{w!}\nonumber\\
&  \cdot\left[  f^{(w-1)}(\beta)-cn(n-1)...(n-w+2)\beta^{n-w+2}\right]
\left(
%TCIMACRO{\dsum \limits_{j=1}^{l}}%
%BeginExpansion
{\displaystyle\sum\limits_{j=1}^{l}}
%EndExpansion
a_{j}(x_{j}-\beta)^{w}-%
%TCIMACRO{\dsum \limits_{i=1}^{m}}%
%BeginExpansion
{\displaystyle\sum\limits_{i=1}^{m}}
%EndExpansion
b_{i}(y_{i}-\beta)^{w}\right) \nonumber\\
&  -\frac{1}{\beta-\alpha}\sum_{w=1}^{n-1}\frac{n-w}{w!}\nonumber\\
&  \cdot\left[  f^{(w-1)}(\alpha)-cn(n-1)...(n-w+2)\alpha^{n-w+2}\right]
\left(
%TCIMACRO{\dsum \limits_{j=1}^{l}}%
%BeginExpansion
{\displaystyle\sum\limits_{j=1}^{l}}
%EndExpansion
a_{j}(x_{j}-\alpha)^{w}-%
%TCIMACRO{\dsum \limits_{i=1}^{m}}%
%BeginExpansion
{\displaystyle\sum\limits_{i=1}^{m}}
%EndExpansion
b_{i}(y_{i}-\alpha)^{w}\right)  .\nonumber
\end{align}
If the reverse inequality in (\ref{Th_Fink2_assump}) holds, then the reverse
inequality in (\ref{Th_Fink2a_rez}) holds.

\begin{proof}
Let us consider the function $g(x)=f(x)-cx^{n}.$ Since $f$ is strongly
$n$-convex with modulus $c,$ then $g$ is $n$-convex. We may assume without
loss of generality that $f$ and $g$ are $n$-times differentiable and
$g^{(n)}\geq0$ on $[\alpha,\beta]$ (see \cite[p. 16]{PPT})$.$ \newline
Applying (\ref{Th_Fink2a_rez}) to $g,$ we have%
\begin{align}
&
%TCIMACRO{\dsum \limits_{j=1}^{l}}%
%BeginExpansion
{\displaystyle\sum\limits_{j=1}^{l}}
%EndExpansion
a_{j}g(x_{j})-%
%TCIMACRO{\dsum \limits_{i=1}^{m}}%
%BeginExpansion
{\displaystyle\sum\limits_{i=1}^{m}}
%EndExpansion
b_{i}g(y_{i})\label{rez0}\\
&  =\frac{1}{\beta-\alpha}\sum_{w=2}^{n-1}\frac{n-w}{w!}\cdot g^{(w-1)}%
(\beta)\left(
%TCIMACRO{\dsum \limits_{j=1}^{l}}%
%BeginExpansion
{\displaystyle\sum\limits_{j=1}^{l}}
%EndExpansion
a_{j}(x_{j}-\beta)^{w}-%
%TCIMACRO{\dsum \limits_{i=1}^{m}}%
%BeginExpansion
{\displaystyle\sum\limits_{i=1}^{m}}
%EndExpansion
b_{i}(y_{i}-\beta)^{w}\right) \nonumber\\
&  -\frac{1}{\beta-\alpha}\sum_{w=2}^{n-1}\frac{n-w}{w!}\cdot g^{(w-1)}%
(\alpha)\left(
%TCIMACRO{\dsum \limits_{j=1}^{l}}%
%BeginExpansion
{\displaystyle\sum\limits_{j=1}^{l}}
%EndExpansion
a_{j}(x_{j}-\alpha)^{w}-%
%TCIMACRO{\dsum \limits_{i=1}^{m}}%
%BeginExpansion
{\displaystyle\sum\limits_{i=1}^{m}}
%EndExpansion
b_{i}(y_{i}-\alpha)^{w}\right) \nonumber\\
&  +\frac{1}{(n-1)!(\beta-\alpha)}\times\nonumber\\
&
%TCIMACRO{\dint \nolimits_{\alpha}^{\beta}}%
%BeginExpansion
{\displaystyle\int\nolimits_{\alpha}^{\beta}}
%EndExpansion
\left[
%TCIMACRO{\dsum \limits_{j=1}^{l}}%
%BeginExpansion
{\displaystyle\sum\limits_{j=1}^{l}}
%EndExpansion
a_{j}(x_{j}-t)^{n-1}k(t,x_{j})-%
%TCIMACRO{\dsum \limits_{i=1}^{m}}%
%BeginExpansion
{\displaystyle\sum\limits_{i=1}^{m}}
%EndExpansion
b_{i}(y_{i}-t)^{n-1}k(t,y_{i})\right]  g^{(n)}(t)dt.\nonumber
\end{align}
Moreover, if (\ref{Th_Fink2_assump}) holds, then%
\begin{align}
&
%TCIMACRO{\dsum \limits_{j=1}^{l}}%
%BeginExpansion
{\displaystyle\sum\limits_{j=1}^{l}}
%EndExpansion
a_{j}g(x_{j})-%
%TCIMACRO{\dsum \limits_{i=1}^{m}}%
%BeginExpansion
{\displaystyle\sum\limits_{i=1}^{m}}
%EndExpansion
b_{i}g(y_{i})\label{rez0b}\\
&  \geq\frac{1}{\beta-\alpha}\sum_{w=2}^{n-1}\frac{n-w}{w!}\cdot
g^{(w-1)}(\beta)\left(
%TCIMACRO{\dsum \limits_{j=1}^{l}}%
%BeginExpansion
{\displaystyle\sum\limits_{j=1}^{l}}
%EndExpansion
a_{j}(x_{j}-\beta)^{w}-%
%TCIMACRO{\dsum \limits_{i=1}^{m}}%
%BeginExpansion
{\displaystyle\sum\limits_{i=1}^{m}}
%EndExpansion
b_{i}(y_{i}-\beta)^{w}\right) \nonumber\\
&  -\frac{1}{\beta-\alpha}\sum_{w=2}^{n-1}\frac{n-w}{w!}\cdot g^{(w-1)}%
(\alpha)\left(
%TCIMACRO{\dsum \limits_{j=1}^{l}}%
%BeginExpansion
{\displaystyle\sum\limits_{j=1}^{l}}
%EndExpansion
a_{j}(x_{j}-\alpha)^{w}-%
%TCIMACRO{\dsum \limits_{i=1}^{m}}%
%BeginExpansion
{\displaystyle\sum\limits_{i=1}^{m}}
%EndExpansion
b_{i}(y_{i}-\alpha)^{w}\right) \nonumber
\end{align}
which is equivalent to (\ref{Th_Fink2a_rez}).\newline If the reverse
inequality in (\ref{Th_Fink2_assump}) holds, then the last term in
(\ref{rez0}) is nonpositive and then the reverse inequality in (\ref{rez0b})
holds. This ends the proof.
\end{proof}
\end{theorem}

\begin{remark}
Consider the function $s:[\alpha,\beta]\rightarrow\mathbb{R}$ defined by%
\[
s(x)=(x-t)^{n-1}k(t,x)=\left\{
\begin{tabular}
[c]{ll}%
$(x-t)^{n-1}(t-\alpha),$ & $\alpha\leq t\leq x\leq\beta$\\
$(x-t)^{n-1}(t-\beta),$ & $\alpha\leq x<t\leq\beta$%
\end{tabular}
\ \ \ \ \right.  .
\]
We have%
\[
s^{\prime\prime}(x)=\left\{
\begin{tabular}
[c]{ll}%
$(n-1)(n-2)(x-t)^{n-3}(t-\alpha),$ & $\alpha\leq t\leq x\leq\beta$\\
$(n-1)(n-2)(x-t)^{n-3}(t-\beta),$ & $\alpha\leq x<t\leq\beta$%
\end{tabular}
\ \ \ \ \ \ \ \right.  .
\]
Then for even $n,$ the function $s$ is convex and by Sherman's theorem, we
have%
\[%
%TCIMACRO{\dsum \limits_{j=1}^{l}}%
%BeginExpansion
{\displaystyle\sum\limits_{j=1}^{l}}
%EndExpansion
a_{j}(x_{j}-t)^{n-1}k(t,x_{j})-%
%TCIMACRO{\dsum \limits_{i=1}^{m}}%
%BeginExpansion
{\displaystyle\sum\limits_{i=1}^{m}}
%EndExpansion
b_{i}(y_{i}-t)^{n-1}k(t,y_{i})\geq0,
\]
i.e. the assumption (\ref{Th_Fink2_assump}) is immediately satisfied.
Therefore, by Theorem \ref{Th_Fink2a}, the inequality (\ref{Th_Fink2a_rez})
holds. \newline Specially, for $n=2,$ the inequality (\ref{Th_Fink2a_rez})
reduces to (\ref{Sh_str}).
\end{remark}


\begin{thebibliography}{99}                                                                                               %


\bibitem {AKIBP1}M. Adil Khan, S. Iveli\'{c} Bradanovi\'{c} and J.
Pe\v{c}ari\'{c}, Generalizations of Sherman's inequality by Hermite's
interpolating polynomial, Math. Inequal. Appl. 19 (4) (2016), 1181-1192.

\bibitem {AKIBP2}M. Adil Khan, S. Iveli\'{c} Bradanovi\'{c} and J.
Pe\v{c}ari\'{c}, Generalizations of Sherman's inequality by Hermite's
interpolating polynomial and Green function, Konuralp J. Math. 4(2) (2016), 255-270.

\bibitem {AIBP}R. P. Agarwal, S. Iveli\'{c} Bradanovi\'{c} and J.
Pe\v{c}ari\'{c}, Generalizations of Sherman's inequality by Lidstone's
interpolating polynomial, J. Inequal. Appl. 6, 2016 (2016)

\bibitem {B}J. Borcea, Equilibrium points of logarithmic potentials, Trans.
Amer. Math. Soc., 359 (2007) 3209-3237.

\bibitem {Bu}A. M. Burtea , Two examples of weighted majorization, An. Univ.
Craiova Ser. Mat. Inform. 37 (2), 2010, 92--99.

\bibitem {CSIS}I. Csisz\'{a}r, Information-type measures of difference of
probability functions and indirect observations, Studia Sci. Math. Hungar, 2
(1967), 299-318.

\bibitem {CK}I. Csisz\'{a}r and J. K\"{o}rner, Information Theory: Coding
Theorem for Discrete Memoryless Systems, Academic Press, New York, 1981.

\bibitem {Drag}S. S. Dragomir, Upper and lower bounds for Csiszar%
%TCIMACRO{\U{b4} }%
%BeginExpansion
\'{}
%EndExpansion
f-divergence in terms of the Kullback-Leibler divergence and applications,
Inequalities for Csiszar f-Divergence in Information Theory, RGMIA Monographs,
Victoria University, Australia, 2000.

\bibitem {FINK}A. M. Fink, Bounds of the deviation of a function from its
averages, Czechoslovak Math. J. 42 (117) (1992), 289-310.

\bibitem {F}L. Fuchs, A new proof of an inequality of Hardy-Littlewood-Polya,
Mat. Tidsskr. B. 1947 (1947), 53--54.

\bibitem {GN}R. Gera, K. Nikodem, Strongly convex functions of higher order,
Nonlinear Anal. 74, 2011, 661--665.

\bibitem {IBNP}S. Iveli\'{c} Bradanovi\'{c}, N. Latif, J. Pe\v{c}ari\'{c}, On
an upper bound for Sherman's inequality. J. Inequal. Appl. 2016 (2016).

\bibitem {IBNP2}S. Iveli\'{c} Bradanovi\'{c}, N. Latif, \DJ . Pe\v{c}ari\'{c},
J. Pe\v{c}ari\'{c}, Sherman's and related inequalities with applications in
information theory, J. Inequal. Appl. 2018 (2018).

\bibitem {IBP2}S. Iveli\'{c} Bradanovi\'{c}, J. Pe\v{c}ari\'{c}, Extensions
and improvements of Sherman's and related inequalities for $n$-convex
functions, Open Math. 15 (1) 2017.

\bibitem {IBP}S. Iveli\'{c} Bradanovi\'{c}, J. Pe\v{c}ari\'{c},
Generalizations of Sherman's inequality, Per. Math. Hung. 74 (2) 2017.

\bibitem {Kap}J. N. Kapur, A comparative assessment of various measures of
directed divergence, Advances in Management Studies, 3 (1984), No. 1, 1-16.

\bibitem {Kap1}J. N. Kapur, Maximum-Entropy Models in Science and Engineering,
John Wiley and Sons, New York (1989).

\bibitem {Kull}S. Kullback, Information Theory and Statistics, Wiley, New York 1959.

\bibitem {LR}P. Lah and M. Ribari\v{c}, Converse of Jensen's inequality for
convex functions, Univ. Beograd Publ. Elektrotehn. Fak. Ser. Mat. Fiz.
412--460 (1973), 201--205.

\bibitem {HLP}G. H. Hardy, J. E. Littlewood, G. P\'{o}lya, Inequalities,
Cambridge University Press, 2nd ed., Cambridge 1952.

\bibitem {MO}A. W. Marshall, I. Olkin and Barry C. Arnold, Inequalities:
Theory of Majorization and Its Applications (Second Edition), Springer Series
in Statistics, New York 2011.

\bibitem {MN}N. Merentes, K. Nikodem, Some remarks on strongly convex
functions, Aequat. Math. 80 (2010), 193--199.

\bibitem {NI}C. P. Niculescu, Choquet theory for signed measures, Math.
Inequal. Appl., 5 (2002), 479--489.

\bibitem {NP1}C. P. Niculescu, L.-E. Persson, Old and New on the
Hermite-Hadamard inequality, Real Analysis Exchange, 2004.

\bibitem {NR}C. P. Niculescu, I. Roven\c{t}a, An approach of majorization in
spaces with a curved geometry, J. Math. Anal. Appl. 411 (2014) 119--128

\bibitem {NR2}C. P. Niculescu, I. Roven\c{t}a, Relative convexity and its
applications, Aequat. Math. 89 (2015), 1389--1400

\bibitem {MN0}M. Niezgoda, Nonlinear Sherman-type inequalities, Advances in
Nonlinear Analysis, DOI: https://doi.org/10.1515/anona-2018-0098

\bibitem {MN1}M. Niezgoda, Remarks on Sherman like inequalities for
$(\alpha,\beta)$-convex functions, Math. Ineqal. Appl. 17 (4) (2014) 1579--1590.

\bibitem {MN2}M. Niezgoda, Remarks on convex functions and separable
sequences, Discrete Math. 308 (2008), 1765-1773.

\bibitem {MN3}M. Niezgoda, Vector joint majorization and generalization of
Csiszar-Korner's inequality for f-divergence, Discrete Appl. Math., 198
(2016), 195-205.

\bibitem {N}K. Nikodem, On Strongly Convex Functions and Related Classes of
Functions, Handbook of Functional Equations. Springer New York (2014), 365--405.

\bibitem {NP}K. Nikodem, Zs. P\'{a}les, Characterizations of inner product
spaces by strongly convex functions, Banach J. Math. Anal. 5 (2011), no. 1, 83--87.

\bibitem {PPT}J. Pe\v{c}ari\'{c}, F. Proschan and Y. L. Tong, Convex
functions, Partial Orderings and Statistical Applications, Academic Press, New
York, 1992.

\bibitem {P}B.T. Polyak, Existence theorems and convergence of minimizing
sequences in extremum problems with restrictions, Soviet Math. Dokl. 7 (1966), 72--75.

\bibitem {POP}T. Popoviciu, Les Fonctions Convexes, Hermann et Cie, Paris, 1944.

\bibitem {Shann}C. E. Shannon and W. Weaver, The Mathemtiatical Theory of
Comnunication, Urbana, University of Illinois Press, 1949.

\bibitem {S}S. Sherman, On a theorem of Hardy, Littlewood, P\'{o}lya and
Blackwell, Proc. Nat. Acad. Sci. USA, 37 (1) (1957), 826-831.
\end{thebibliography}
\end{document}